\documentclass{article}
\usepackage{amsmath}
\usepackage{amsfonts}

\newtheorem{theorem}{Theorem}

\newtheorem{definition}[theorem]{Definition}
\newtheorem{example}[theorem]{Example}

\newtheorem{lemma}[theorem]{Lemma}

\newtheorem{proposition}[theorem]{Proposition}
\newtheorem{remark}[theorem]{Remark}

\newenvironment{proof}[1][Proof]{\noindent\textbf{#1.} }{\ \rule{0.5em}{0.5em}}
\input{tcilatex}
\begin{document}

\title{\textbf{Hyperbolic Valued Metric Space}}
\author{Chinmay Ghosh$^{1}$, Anirban Bandyopadhyay$^{2}$, Soumen Mondal$^{3}$%
\qquad \\
%EndAName
$^{1}$Department of Mathematics\\
Kazi Nazrul University\\
Nazrul Road, P.O.- Kalla C.H.\\
Asansol-713340, West Bengal, India \\
chinmayarp@gmail.com\\
$^{2}$53, Gopalpur Primary School\\
Raninagar-I, Murshidabad\\
Pin-742304\\
West Bengal, India \\
anirbanbanerjee010@gmail.com \\
$^{3}$28, Dolua Dakshinpara Haridas Primary School\\
Beldanga, Murshidabad\\
Pin-742133\\
West Bengal, India\\
mondalsoumen79@gmail.com}
\date{}
\maketitle

\begin{abstract}
In this paper our main aim is to develop some basic properties of hyperbolic
valued metric spaces. We also establish the hyperbolic version of Banach
contraction principle. Further we construct a hyperbolic valued metric on
the space of all hyperbolic valued continuous functions and prove some
results.


\end{abstract}

\section{\protect\bigskip Introductions and Preliminaries}

\bigskip Corrado Segre \cite{seg} introduced bicomplex numbers in 1892. The
set of bicomplex numbers \cite{alp}, \cite{sai} is defined as 
\begin{equation*}
\mathbb{B%
%TCIMACRO{\U{2102} }%
%BeginExpansion
\mathbb{C}
%EndExpansion
}=\{w=w_{1}+\mathbf{j}w_{2}:w_{1},w_{2}\in 
%TCIMACRO{\U{2102} }%
%BeginExpansion
\mathbb{C}
%EndExpansion
(\mathbf{i})\},
\end{equation*}%
where $\mathbf{i}$ and $\mathbf{j}$ are imaginary units such that $\mathbf{ij%
}=\mathbf{ji},~\mathbf{i}^{2}=\mathbf{j}^{2}=-1$ and $\mathbb{%
%TCIMACRO{\U{2102} }%
%BeginExpansion
\mathbb{C}
%EndExpansion
}(\mathbf{i})$ is the set of complex numbers with imaginary unit $\mathbf{i}%
. $ Also $\mathbb{%
%TCIMACRO{\U{2102} }%
%BeginExpansion
\mathbb{C}
%EndExpansion
}(\mathbf{j})$ is the set of complex numbers with imaginary unit\textbf{\ }$%
\mathbf{j}.$ Throughout this paper we take $%
%TCIMACRO{\U{211d} }%
%BeginExpansion
\mathbb{R}
%EndExpansion
,\mathbb{%
%TCIMACRO{\U{2102} }%
%BeginExpansion
\mathbb{C}
%EndExpansion
}$ to be fields of real and complex numbers respectively and $\mathbb{N}$ to
be the set of natural numbers. Note that basically the three sets $\mathbb{%
%TCIMACRO{\U{2102} }%
%BeginExpansion
\mathbb{C}
%EndExpansion
},\mathbb{%
%TCIMACRO{\U{2102} }%
%BeginExpansion
\mathbb{C}
%EndExpansion
}(\mathbf{i}),\mathbb{%
%TCIMACRO{\U{2102} }%
%BeginExpansion
\mathbb{C}
%EndExpansion
}(\mathbf{j})$ are same. As $\mathbb{%
%TCIMACRO{\U{2102} }%
%BeginExpansion
\mathbb{C}
%EndExpansion
}$ consists of two copies of $%
%TCIMACRO{\U{211d} }%
%BeginExpansion
\mathbb{R}
%EndExpansion
,$ called the real and imaginary line, $\mathbb{B%
%TCIMACRO{\U{2102} }%
%BeginExpansion
\mathbb{C}
%EndExpansion
}$ also consists of two copies of $\mathbb{%
%TCIMACRO{\U{2102} }%
%BeginExpansion
\mathbb{C}
%EndExpansion
}$, called $\mathbb{%
%TCIMACRO{\U{2102} }%
%BeginExpansion
\mathbb{C}
%EndExpansion
}(\mathbf{i})$ and $\mathbb{%
%TCIMACRO{\U{2102} }%
%BeginExpansion
\mathbb{C}
%EndExpansion
}(\mathbf{j}).$ The set $\mathbb{B%
%TCIMACRO{\U{2102} }%
%BeginExpansion
\mathbb{C}
%EndExpansion
}$ forms a commutative ring with unity and with zero divisors under the
usual addition and multiplication of bicomplex numbers. The product of
imaginary units $\mathbf{i}$ and $\mathbf{j}$ defines a hyperbolic unit $%
\mathbf{k}$ such that $\mathbf{k}^{2}=1.$ The three units satisfy%
\begin{equation*}
\mathbf{ij}=\mathbf{ji}=\mathbf{k},~\mathbf{ik}=\mathbf{ki}=-\mathbf{j}\text{
and }\mathbf{jk}=\mathbf{kj}=-\mathbf{i}.
\end{equation*}

Along with two complex planes $\mathbb{%
%TCIMACRO{\U{2102} }%
%BeginExpansion
\mathbb{C}
%EndExpansion
}(\mathbf{i})$ and $\mathbb{%
%TCIMACRO{\U{2102} }%
%BeginExpansion
\mathbb{C}
%EndExpansion
}(\mathbf{j})$, $\mathbb{B%
%TCIMACRO{\U{2102} }%
%BeginExpansion
\mathbb{C}
%EndExpansion
}$ contains a third interesting plane called the hyperbolic plane. The
elements of the hyperbolic plane are the bicomplex numbers like $\mathbb{%
\alpha =}a_{1}+\mathbf{k}a_{2}$ where $a_{1},a_{2}\in \mathbb{%
%TCIMACRO{\U{211d} }%
%BeginExpansion
\mathbb{R}
%EndExpansion
}$. The set of hyperbolic numbers is denoted by $\mathbb{D}$ i.e.,%
\begin{equation*}
\mathbb{D=\{\alpha =}a_{1}+\mathbf{k}a_{2}:a_{1},a_{2}\in \mathbb{%
%TCIMACRO{\U{211d} }%
%BeginExpansion
\mathbb{R}
%EndExpansion
\}}.
\end{equation*}%
Note that $\mathbb{D}$\ is a commutative ring with unity and with zero
divisors under the usual addition and multiplication of $\mathbb{B%
%TCIMACRO{\U{2102} }%
%BeginExpansion
\mathbb{C}
%EndExpansion
}$.

The two zero divisors 
\begin{equation*}
\mathbf{e}_{1}=\frac{1+\mathbf{k}}{2}\text{ and }\mathbf{e}_{2}=\frac{1-%
\mathbf{k}}{2}
\end{equation*}%
form an idempotent basis in $\mathbb{B%
%TCIMACRO{\U{2102} }%
%BeginExpansion
\mathbb{C}
%EndExpansion
}.$

If $\mathbb{\alpha =}a_{1}+\mathbf{k}a_{2}\in \mathbb{D},$ then the
idempotent representation of $\mathbb{\alpha }$ is%
\begin{equation*}
\mathbb{\alpha =(}a_{1}+a_{2})\mathbf{e}_{1}+\mathbb{(}a_{1}-a_{2})\mathbf{e}%
_{2}.
\end{equation*}

Sum and the product of two hyperbolic numbers can be defined pointwise with
the idempotent basis.

There exists a bijection between the Euclidean plane $%
%TCIMACRO{\U{211d} }%
%BeginExpansion
\mathbb{R}
%EndExpansion
^{2}$ and $\mathbb{D},$ in which every hyperbolic number $\mathbb{\alpha =}%
\alpha _{1}\mathbf{e}_{1}+\alpha _{2}\mathbf{e}_{2}$ is mapped to $(\alpha
_{1},\alpha _{2})\in 
%TCIMACRO{\U{211d} }%
%BeginExpansion
\mathbb{R}
%EndExpansion
^{2}$ and vice versa.

Also we have the two projection maps $p_{1},p_{2}:\mathbb{D\rightarrow }%
%TCIMACRO{\U{211d} }%
%BeginExpansion
\mathbb{R}
%EndExpansion
,$ given by%
\begin{equation*}
p_{i}(\mathbb{\alpha })=\alpha _{i},\text{ where }\mathbb{\alpha =}\alpha
_{1}\mathbf{e}_{1}+\alpha _{2}\mathbf{e}_{2},~i=1,2.
\end{equation*}

We say that a hyperbolic number $\mathbb{\alpha =}\alpha _{1}\mathbf{e}%
_{1}+\alpha _{2}\mathbf{e}_{2}\in \mathbb{D}$ is a positive hyperbolic
number if $\alpha _{1},\alpha _{2}>0.$ Thus the set of positive hyperbolic
numbers $\mathbb{D}^{+}$ is given by%
\begin{equation*}
\mathbb{D}^{+}\mathbb{=\{\alpha =}\alpha _{1}\mathbf{e}_{1}+\alpha _{2}%
\mathbf{e}_{2}:\alpha _{1},\alpha _{2}>0\mathbb{\}},
\end{equation*}%
and we denote the set of all zero divisors in $\mathbb{D}$ by $\mathbb{O}.$
Set $\mathbb{O}_{0}=\mathbb{O}\cup \{0\}.$\ And we use the notation $\mathbb{%
D}_{0}^{+}=\mathbb{D}^{+}\cup \mathbb{O}_{0}.$

For $\alpha ,\beta \in \mathbb{D},~$define~a relation $\preceq $ on \bigskip 
$\mathbb{D}$~\cite{ku}$\ $by $\alpha \preceq \beta $ whenever $\beta -\alpha
\in \mathbb{D}_{0}^{+}.$ This relation is reflexive, anti-symmetric as well
as transitive and hence defines a partial order on $\mathbb{D}.$ If we write
the hyperbolic numbers $\alpha ,\beta $ in idempotent representation as $%
\mathbb{\alpha =}\alpha _{1}\mathbf{e}_{1}+\alpha _{2}\mathbf{e}_{2}$ and $%
\mathbb{\beta =\beta }_{1}\mathbf{e}_{1}+\mathbb{\beta }_{2}\mathbf{e}_{2},$
then $\alpha \preceq \beta $ implies that $\alpha _{1}\leq \beta _{1}$ and $%
\alpha _{2}\leq \beta _{2}.$ And by $\alpha \prec \beta $ we mean $\alpha
_{1}<\beta _{1}$ and $\alpha _{2}<\beta _{2}.$

For $A\subset \mathbb{D},$ define \cite{tell},%
\begin{eqnarray*}
A_{\mathbf{e}_{1}} &=&\{x\in 
%TCIMACRO{\U{211d} }%
%BeginExpansion
\mathbb{R}
%EndExpansion
:\exists y\in 
%TCIMACRO{\U{211d} }%
%BeginExpansion
\mathbb{R}
%EndExpansion
\text{ such that }x\mathbf{e}_{1}+y\mathbf{e}_{2}\in A\}, \\
A_{\mathbf{e}_{2}} &=&\{y\in 
%TCIMACRO{\U{211d} }%
%BeginExpansion
\mathbb{R}
%EndExpansion
:\exists x\in 
%TCIMACRO{\U{211d} }%
%BeginExpansion
\mathbb{R}
%EndExpansion
\text{ such that }x\mathbf{e}_{1}+y\mathbf{e}_{2}\in A\}.
\end{eqnarray*}%
and we will consider the supremum and infimum of $A$ defined as follows%
\begin{eqnarray*}
\sup A &=&\sup A_{\mathbf{e}_{1}}\mathbf{e}_{1}+\sup A_{\mathbf{e}_{2}}%
\mathbf{e}_{2}, \\
\inf A &=&\inf A_{\mathbf{e}_{1}}\mathbf{e}_{1}+\inf A_{\mathbf{e}_{2}}%
\mathbf{e}_{2}.
\end{eqnarray*}

\bigskip Moreover, any bicomplex number $w=w_{1}+jw_{2}$ can also be written
as 
\begin{equation*}
w=z_{1}\mathbf{e}_{1}+z_{2}\mathbf{e}_{2}\text{ where }z_{1}=w_{1}-\mathbf{i}%
w_{2},~z_{2}=w_{1}+\mathbf{i}w_{2}\in 
%TCIMACRO{\U{2102} }%
%BeginExpansion
\mathbb{C}
%EndExpansion
(\mathbf{i}).
\end{equation*}

The hyperbolic modulus of a bicomplex number is defined by%
\begin{equation*}
\left\vert w\right\vert _{k}=\left\vert z_{1}\mathbf{e}_{1}+z_{2}\mathbf{e}%
_{2}\right\vert _{k}=\left\vert z_{1}\right\vert \mathbf{e}_{1}+\left\vert
z_{2}\right\vert \mathbf{e}_{2}\in \mathbb{D}_{0}^{+}
\end{equation*}%
where $\left\vert .\right\vert $ is the real modulus of a complex number$.$

In this paper our main aim is to develop some basic properties of hyperbolic
valued metric spaces. We also establish the hyperbolic version of Banach
contraction principle. Further we construct a hyperbolic valued metric on
the space of all hyperbolic valued continuous functions and prove some
results.

\section{Basic definitions and lemmas}

In this section we give some basic definitions and lemmas. We rename some
definitions already known.

\begin{definition}[$\protect\cite{ku}$]
Let $X$ be a nonempty set. A function $d_{\mathbb{D}}:X\times
X\longrightarrow \mathbb{D}_{0}^{+}$ satisfying the following:

$(i)$ $d_{\mathbb{D}}(x,y)\succeq 0$ for all $x,y\in X$ and $d_{\mathbb{D}%
}(x,y)=0$ if and only if $x=y,$

$(ii)$ $d_{\mathbb{D}}(x,y)=d_{\mathbb{D}}(y,x)$ for all $x,y\in X,$

$(iii)$ $d_{\mathbb{D}}(x,y)\preceq d_{\mathbb{D}}(x,z)+d_{\mathbb{D}}(z,y)$
for all $x,y,z\in X,$

is defined to be a hyperbolic valued metric or $\mathbb{D}$-metric on $X$
and $(X,d_{\mathbb{D}})$ is called a hyperbolic valued metric space or $%
\mathbb{D}$-metric space.
\end{definition}

Observe that a $\mathbb{D}$-metric $d_{\mathbb{D}}$ on $X$ can be decomposed
as $d_{\mathbb{D}}(x,y)=d_{1}(x,y)\mathbf{e}_{1}+d_{2}(x,y)\mathbf{e}_{2}$
where $d_{1}(x,y)$\textbf{\ }and\textbf{\ }$d_{2}(x,y)$ are two real
metrices on $X.$

\begin{definition}
Let $(X,d_{\mathbb{D}})$ and $(Y,\rho _{\mathbb{D}})$ be two $\mathbb{D}$%
-metric spaces. A $\mathbb{D}$-isometry between $(X,d_{\mathbb{D}})$ to $%
(Y,\rho _{\mathbb{D}})$ is a bijection $i_{\mathbb{D}}:X\longrightarrow Y$
such that $d_{\mathbb{D}}(x,y)=\rho _{\mathbb{D}}(i_{\mathbb{D}}(x),i_{%
\mathbb{D}}(y))$ for all $x,y\in X.$ We say that $(X,d_{\mathbb{D}})$ and $%
(Y,\rho _{\mathbb{D}})$ are $\mathbb{D}$-isometric if there exists an
isometry from $(X,d_{\mathbb{D}})$ to $(Y,\rho _{\mathbb{D}}).$
\end{definition}

We may think that two hyperbolic valued metric spaces as the same if they
are $\mathbb{D}$-isometric. Note that if $i_{\mathbb{D}}$ is a $\mathbb{D}$%
-isometry from $(X,d_{\mathbb{D}})$ and $(Y,\rho _{\mathbb{D}}),$ then the
inverse $i_{\mathbb{D}}^{-1}$ is also a $\mathbb{D}$-isometry from $(Y,\rho
_{\mathbb{D}})$ to $(X,d_{\mathbb{D}}),$ and hence being $\mathbb{D}$%
-isometric is a symmetric relation.

\begin{definition}
Let $(X,d_{\mathbb{D}})$ and $(Y,\rho _{\mathbb{D}})$ be two $\mathbb{D}$%
-metric spaces. An embedding of $(X,d_{\mathbb{D}})$ into $(Y,\rho _{\mathbb{%
D}})$ is an injection $i_{\mathbb{D}}:X\longrightarrow Y$ such that $d_{%
\mathbb{D}}(x,y)=\rho _{\mathbb{D}}(i_{\mathbb{D}}(x),i_{\mathbb{D}}(y))$
for all $x,y\in X.$
\end{definition}

Note that an embedding $i_{\mathbb{D}}$ can be regarded as a $\mathbb{D}$%
-isometry between $X$ and its image $i_{\mathbb{D}}(X).$

\begin{definition}[$\protect\cite{sai}$]
Let $(X,d_{\mathbb{D}})$ be a $\mathbb{D}$-metric space. A sequence $%
\{x_{n}\}$ in $X$ converges to a point $\alpha $ if for every $\epsilon _{%
\mathbb{D}}\in \mathbb{D}^{+}~$there exists an $N\in 
%TCIMACRO{\U{2115} }%
%BeginExpansion
\mathbb{N}
%EndExpansion
$ such that $d_{\mathbb{D}}(x_{n},\alpha )\prec \epsilon _{\mathbb{D}}$ for
all $n\geq N.$ We write $\lim\limits_{n\rightarrow \infty }x_{n}=\alpha $ or 
$x_{n}\rightarrow \alpha $ and $\{x_{n}\}$ is called $\mathbb{D}$-convergent
sequence.
\end{definition}

The following lemma can be easily proved and can be considered as an
alternative definition of $\mathbb{D}$-convergene:

\begin{lemma}
A sequence $\{x_{n}\}$ in a $\mathbb{D}$-metric space $(X,d_{\mathbb{D}})$
converges to $\alpha $ if and only if $\lim\limits_{n\rightarrow \infty }d_{%
\mathbb{D}}(x_{n},\alpha )=0.$
\end{lemma}

\begin{definition}
Let $(X,d_{\mathbb{D}})$ and $(Y,\rho _{\mathbb{D}})$ be two $\mathbb{D}$%
-metric spaces. A function $f:X\longrightarrow Y$ is $\mathbb{D}$-continuous
at a point $\alpha \in X$ if for every $\epsilon _{\mathbb{D}}\in \mathbb{D}%
^{+}$ there exists a $\delta _{\mathbb{D}}\in \mathbb{D}^{+}$ such that $%
\rho _{\mathbb{D}}(f(x),f(y))\prec \epsilon _{\mathbb{D}}$ whenever $d_{%
\mathbb{D}}(x,\alpha )\prec \delta _{\mathbb{D}}.$ The function $f$ is
called $\mathbb{D}$-continuous on $X$ if it is $\mathbb{D}$-continuous at
all points in $X.$
\end{definition}

The following lemma is easy to prove:

\begin{lemma}
Let $(X,d_{\mathbb{D}})$ and $(Y,\rho _{\mathbb{D}})$ be two $\mathbb{D}$%
-metric spaces and $f:X\longrightarrow Y$. Then the following are equivalent:

$i)$ $f$ is $\mathbb{D}$-continuous at a point $\alpha \in X.$

$ii)$ For every sequence $\{x_{n}\}$ in the $\mathbb{D}$-metric space $(X,d_{%
\mathbb{D}})$ converging to $\alpha \in X,~$the sequence $\{f\left(
x_{n}\right) \}$ in the $\mathbb{D}$-metric space $(Y,\rho _{\mathbb{D}})$
converges to $f(\alpha ).$
\end{lemma}

\begin{definition}[$\protect\cite{sai}$]
Let $a$ be a point in a $\mathbb{D}$-metric space $(X,d_{\mathbb{D}}),$ and
assume that $r\in \mathbb{D}^{+}.$ The open $\mathbb{D}$-ball centered at $a$
with hyperbolic radius $r$ is the set 
\begin{equation*}
B_{\mathbb{D}}(a;r)=\{x\in X:d_{\mathbb{D}}(x,a)\prec r\}.
\end{equation*}

The closed $\mathbb{D}$-ball centered at $a$ with hyperbolic radius $r$ is
the set%
\begin{equation*}
\overline{B_{\mathbb{D}}(a;r)}=\{x\in X:d_{\mathbb{D}}(x,a)\preceq r\}.
\end{equation*}

And the $\mathbb{D}$-sphere centered at $a$ with hyperbolic radius $r$ is
the set%
\begin{equation*}
S_{\mathbb{D}}(a;r)=\{x\in X:d_{\mathbb{D}}(x,a)=r\}.
\end{equation*}
\end{definition}

\begin{remark}
Let $\alpha ,\beta \in \mathbb{D}$, with $\alpha =\alpha _{1}\mathbf{e}%
_{1}+\alpha _{2}\mathbf{e}_{2}$ and $\beta =\alpha _{1}\mathbf{e}_{1}+\alpha
_{2}\mathbf{e}_{2}$ be their idempotent representation.~Also let $\alpha
\prec \beta .$

Then the open $\mathbb{D}$-interval $\left( \alpha ,\beta \right) _{\mathbb{D%
}}$ is defined by%
\begin{equation*}
\left( \alpha ,\beta \right) _{\mathbb{D}}=\{x:\alpha \prec x\prec \beta \}
\end{equation*}%
and the closed $\mathbb{D}$-interval $\left[ \alpha ,\beta \right] _{\mathbb{%
D}}$ is defined by%
\begin{equation*}
\left[ \alpha ,\beta \right] _{\mathbb{D}}=\{x:\alpha \preceq x\preceq \beta
\}.
\end{equation*}
\end{remark}

Note that, if $\gamma ,\delta \in \left( \alpha ,\beta \right) _{\mathbb{D}}$
or $\gamma ,\delta \in \left[ \alpha ,\beta \right] _{\mathbb{D}}$ that does
not necessarily mean $\gamma \preceq \delta $ or $\gamma \succeq \delta .$

\begin{definition}[$\protect\cite{sai}$]
Let $(X,d_{\mathbb{D}})$ be a $\mathbb{D}$-metric space and $A\subset
X~(A\neq \phi ).~$Then

$i)$\ a point $\alpha \in A$ is said to be a $\mathbb{D}$-interior point of $%
A$ if there exists a$~B_{\mathbb{D}}(\alpha ;r)~$for some $r\in \mathbb{D}%
^{+}$ such that $B_{\mathbb{D}}(\alpha ;r)\subset A.$

$ii)$\ a point $\alpha \in X$ is said to be a $\mathbb{D}$-limit point of $A$
if for all$~B_{\mathbb{D}}(\alpha ;r)~($ i.e. for all $r\in \mathbb{D}^{+}),$%
~ $A\cap \left( B_{\mathbb{D}}(\alpha ;r)-\{\alpha \}\right) \neq \phi .$
\end{definition}

The set of all $\mathbb{D}$-interior points of $A$ is denoted by $A_{\mathbb{%
D}}^{0}$ and the set of all $\mathbb{D}$-limit points of $A$ is denoted by $%
A_{\mathbb{D}}^{\prime }.$

\begin{definition}[$\protect\cite{sai}$]
Let $(X,d_{\mathbb{D}})$ be a $\mathbb{D}$-metric space and $A\subset
X~(A\neq \phi ).~$Then

$i)$\ $A$ is said to be $\mathbb{D}$-open subset of $X$ if $A=A_{\mathbb{D}%
}^{0}.$

$ii)$\ $A$ is said to be $\mathbb{D}$-closed subset of $X$ if $A\supset A_{%
\mathbb{D}}^{\prime }.$
\end{definition}

\begin{definition}[$\protect\cite{sai}$]
A sequence $\{x_{n}\}$ in a $\mathbb{D}$-metric space $(X,d_{\mathbb{D}})$
is a $\mathbb{D}$-Cauchy sequence if for each $\epsilon _{\mathbb{D}}\succ 0$
there exists an $N\in 
%TCIMACRO{\U{2115} }%
%BeginExpansion
\mathbb{N}
%EndExpansion
$ such that $d_{\mathbb{D}}(x_{n},x_{m})\prec \epsilon _{\mathbb{D}}$
whenever $n,m\geq N.$
\end{definition}

\begin{remark}
\label{R2.3.1} If we write $d_{\mathbb{D}}(x,y)=d_{1}(x,y)\mathbf{e}%
_{1}+d_{2}(x,y)\mathbf{e}_{2},$ for all $x,y\in X,$ then the sequence $%
\{x_{n}\}$ in a $\mathbb{D}$-metric space $(X,d_{\mathbb{D}})$ is a $\mathbb{%
D}$-Cauchy sequence iff $\{x_{n}\}$ is Cauchy sequence in the real metric
spaces $(X,d_{p})$ for $p=1,2.$
\end{remark}

\begin{definition}[$\protect\cite{sai}$]
A $\mathbb{D}$-metric space $(X,d_{\mathbb{D}})$ is called $\mathbb{D}$%
-complete if every $\mathbb{D}$-Cauchy sequence converges to a point in $%
(X,d_{\mathbb{D}})$.
\end{definition}

\begin{definition}
A subset $K$ of a $\mathbb{D}$-metric space $(X,d_{\mathbb{D}})$ is called a 
$\mathbb{D}$-compact set if for every sequence in $K$ has a subsequence
converging to a point in $K.$ The space $(X,d_{\mathbb{D}})$ is $\mathbb{D}$%
-compact if $X$ is $\mathbb{D}$-compact set, i.e., if each sequence in $X$
has a $\mathbb{D}$-convergent subsequence.
\end{definition}

\begin{definition}
A subset $A$ of a $\mathbb{D}$-metric space $(X,d_{\mathbb{D}})$ is $\mathbb{%
D}$-bounded if there is a number $M\in \mathbb{D}$ such that $d_{\mathbb{D}%
}(a,b)\preceq M$ for all $a,b\in A.$
\end{definition}

\begin{definition}
A subset $A$ of a $\mathbb{D}$-metric space $X$ is called totally $\mathbb{D}
$-bounded if for each $\epsilon _{\mathbb{D}}\in \mathbb{D}^{+}$ there
exists finite number of open $\mathbb{D}$-balls $B_{\mathbb{D}%
}(a_{1},\epsilon _{\mathbb{D}}),$ $B_{\mathbb{D}}(a_{2},\epsilon _{\mathbb{D}%
}),$ $\ldots ,$ $B_{\mathbb{D}}(a_{n},\epsilon _{\mathbb{D}})$ with centers
in $A$ and hyperbolic radius $\epsilon _{\mathbb{D}}$ that cover $A$ i.e., 
\begin{equation*}
A\subseteq B_{\mathbb{D}}(a_{1},\epsilon _{\mathbb{D}})\cup B_{\mathbb{D}%
}(a_{2},\epsilon _{\mathbb{D}})\cup ...\cup B_{\mathbb{D}}(a_{n},\epsilon _{%
\mathbb{D}}).
\end{equation*}
\end{definition}

Now we give an alternative description of $\mathbb{D}$-compactness which is
some time more useful if we want to extend the concept of $\mathbb{D}$%
-compactness to even more general spaces where sequences are not always an
efficient tool, and it is better to have a description of $\mathbb{D}$%
-compactness in terms of covering by $\mathbb{D}$-open sets which will be
called $\mathbb{D}$-open cover.

\begin{definition}
A subset $K$ of a $\mathbb{D}$-metric space $(X,d_{\mathbb{D}})$ is called a 
$\mathbb{D}$-compact set if for every $\mathbb{D}$-open covering $\mathcal{O}
$ of $K,$ i.e., 
\begin{equation*}
K\subset \cup \{O:O\in \mathcal{O}\}
\end{equation*}%
where $\mathcal{O}$ is (finite or infinite) collection of $\mathbb{D}$-open
sets of $(X,d_{\mathbb{D}}),~$there exists finite number of $\mathbb{D}$%
-open sets $O_{1},O_{2},\ldots ,O_{n}\in \mathcal{O}$ such that%
\begin{equation*}
K\subset O_{1}\cup O_{2}\cup \ldots \cup O_{n}.
\end{equation*}
\end{definition}

\begin{definition}
Let $(X,d_{\mathbb{D}})$ be $\mathbb{D}$-metric space. A mapping $%
T:X\rightarrow X$ is said to be $\mathbb{D}$-Lipschitzian if there exists a
constant $k\in \mathbb{D}^{+}$ such that for all $x,y\in X$%
\begin{equation*}
d_{\mathbb{D}}(Tx,Ty)\preceq kd_{\mathbb{D}}(x,y).
\end{equation*}
\end{definition}

The smallest number $k$ satisfying the above is called $\mathbb{D}$%
-Lipschitz constant of $T.$

\begin{definition}
A $\mathbb{D}$-Lipschitzian mapping $T:X\rightarrow X$ with $\mathbb{D}$%
-Lipschitz constant $0\prec k\prec 1$ is said to be a $\mathbb{D}$%
-contraction mapping.
\end{definition}

\begin{remark}
\label{R2.4.1} Clearly every $\mathbb{D}$-Lipschitzian mapping on $\mathbb{D}
$-metric space $(X,d_{\mathbb{D}})$ is $\mathbb{D}$-continuous.
\end{remark}

\begin{definition}
A mapping $T:X\rightarrow X$ is said to be $\mathbb{D}$-contractive if 
\begin{equation*}
d_{\mathbb{D}}(T(x),T(y))\prec d_{\mathbb{D}}(x,y).
\end{equation*}
\end{definition}

\section{Main results}

In this section we prove our main results. We check the validity of some of
the results by examples given in the next section.

\begin{proposition}
\label{P1} Assume that $(X,d_{\mathbb{D}})$ and $(Y,\rho _{\mathbb{D}})$ be
two $\mathbb{D}$-metric spaces and $f:X\longrightarrow Y$. For a point $%
\alpha \in X,$ the following two conditions are equivalent:

i)\ $f$ is $\mathbb{D}$-continuous at $\alpha .~$\ 

ii)\ For all $B_{\mathbb{D}}(f(\alpha );r)\subset Y~(r\in \mathbb{D}^{+}),$
there exists a $B_{\mathbb{D}}(\alpha ;r^{\prime })\subset X~(r^{\prime }\in 
\mathbb{D}^{+})$ such that $f\left( B_{\mathbb{D}}(\alpha ;r^{\prime
})\right) \subset B_{\mathbb{D}}(f(\alpha );r).$
\end{proposition}

\begin{proof}
i)$\Longrightarrow $ii):\ Assume that $f$ is $\mathbb{D}$-continuous at $%
\alpha .$ Clearly, $B_{\mathbb{D}}(f(\alpha );r)$ is an open $\mathbb{D}$%
-ball centered at $f(\alpha )$ in $(Y,\rho _{\mathbb{D}}).$ Since $f$ is $%
\mathbb{D}$-continuous at $\alpha ,$ there is an $r^{\prime }\succ 0$ such
that 
\begin{equation*}
\rho _{\mathbb{D}}(f(x),f(\alpha ))\prec r\text{ whenever }d_{\mathbb{D}%
}(x,\alpha )\prec r^{\prime }.
\end{equation*}%
This means that $f\left( B_{\mathbb{D}}(\alpha ;r^{\prime })\right)
\subseteq B_{\mathbb{D}}(f(\alpha );r).$

ii)$\Longrightarrow $i):\ Now $B_{\mathbb{D}}(f(\alpha );\epsilon _{\mathbb{D%
}})$ $(\epsilon _{\mathbb{D}}\in \mathbb{D}^{+})$ is an open $\mathbb{D}$%
-ball centered at $f(\alpha ),$ then by ii), there exists an open $\mathbb{D}
$-ball $B_{\mathbb{D}}(\alpha ;\delta _{\mathbb{D}})\subset X~(\delta _{%
\mathbb{D}}\in \mathbb{D}^{+})$ such that 
\begin{equation*}
f\left( B_{\mathbb{D}}(\alpha ;\delta _{\mathbb{D}})\right) \subset B_{%
\mathbb{D}}(f(\alpha );\epsilon _{\mathbb{D}})
\end{equation*}%
which implies that $\rho _{\mathbb{D}}(f(x),f(\alpha ))\prec \epsilon _{%
\mathbb{D}}$ whenever $d_{\mathbb{D}}(x,\alpha )\prec \delta _{\mathbb{D}}.$
So $f$ is $\mathbb{D}$-continuous at $\alpha .$
\end{proof}

\begin{proposition}
Assume that $(X,d_{\mathbb{D}})$ and $(Y,\rho _{\mathbb{D}})$ be two $%
\mathbb{D}$-metric spaces and $f:X\longrightarrow Y$. Then the following two
conditions are equivalent:

i)\ $f$ is $\mathbb{D}$-continuous$.$

ii)\ Whenever $V$ is a $\mathbb{D}$-open subset of $Y,$ the inverse image $%
f^{-1}(V)$ is a $\mathbb{D}$-open set in $X.$

\begin{proof}
i)$\Longrightarrow $ii): Assume that $f$ is $\mathbb{D}$-continuous and $%
V\subset Y$ is $\mathbb{D}$-open. We prove that $f^{-1}(V)$ is $\mathbb{D}$%
-open. For any $\alpha \in $ $f^{-1}(V),$ $f(\alpha )\in V,$ and we know
that there is a $\mathbb{D}$-open ball $U$ centered at $\alpha $ such that $%
f(U)\subset V.$ So $U\subset f^{-1}(V),$ and since $\alpha \ $\ is
arbitrary, all the points of $f^{-1}(V)$ are interior points. Hence, $%
f^{-1}(V)$ is $\mathbb{D}$-open.

ii)$\Longrightarrow $i):\ Assume that the inverse images of $\mathbb{D}$%
-open sets are $\mathbb{D}$-open. To prove that $f$ is $\mathbb{D}$%
-continuous at an arbitrary point $\alpha ,$ Proposition \ref{P1} tells us
that it suffices to show that for any $\mathbb{D}$-open ball $B_{\mathbb{D}%
}(f(\alpha );r)\subset Y~(r\in \mathbb{D}^{+}),$ there is a $\mathbb{D}$%
-open ball $B_{\mathbb{D}}(\alpha ;r^{\prime })\subset X~(r^{\prime }\in 
\mathbb{D}^{+})$ such that $f\left( B_{\mathbb{D}}(\alpha ;r^{\prime
})\right) \subset B_{\mathbb{D}}(f(\alpha );r).$ But this is easy: Since the
inverse image of an open set is open, we can simply choose $U=f^{-1}(B_{%
\mathbb{D}}(f(\alpha );r)).$
\end{proof}
\end{proposition}

Note that, the above proposition still holds if we replace $\mathbb{D}$-open
sets by $\mathbb{D}$-closed subsets, since the inverse image commutes with
complements and the $\mathbb{D}$-closed sets are the complements of $\mathbb{%
D}$-open sets.

\begin{proposition}
\label{P2.3.1}Let $\{x_{n}\}$ be a sequence in a $\mathbb{D}$-metric space $%
(X,d_{\mathbb{D}})$ for which%
\begin{equation*}
\sum\limits_{i=1}^{\infty }d_{\mathbb{D}}(x_{i},x_{i+1})\prec \infty _{%
\mathbb{D}},
\end{equation*}
where $\infty _{\mathbb{D}}=\infty \mathbf{e}_{1}+\infty \mathbf{e}_{2}$ 
\cite{gh}. Then $\{x_{n}\}$ is a $\mathbb{D}$-Cauchy sequence in $(X,d_{%
\mathbb{D}})$.

\begin{proof}
We write, $d_{\mathbb{D}}(x,y)=d_{1}(x,y)\mathbf{e}_{1}+d_{2}(x,y)\mathbf{e}%
_{2},$ then $d_{1}(x,y)$ and $d_{2}(x,y)$ are real metrices on $X.$.

Now $\sum\limits_{i=1}^{\infty }d_{\mathbb{D}}(x_{i},x_{i+1})\prec \infty _{%
\mathbb{D}}\Longrightarrow \sum\limits_{i=1}^{\infty
}d_{1}(x_{i},x_{i+1})<\infty $ and $\sum\limits_{i=1}^{\infty
}d_{2}(x_{i},x_{i+1})<\infty .$

Then $\{x_{n}\}$ is Cauchy sequence in the real metric space $(X,d_{p})$ for 
$p=1,2$ [\cite{kh}, Exercise 2.14] and hence by Remark \ref{R2.3.1}, $%
\{x_{n}\}$ is $\mathbb{D}$-Cauchy sequence in $(X,d_{\mathbb{D}}).$
\end{proof}
\end{proposition}

\begin{proposition}
\label{p cplt} Assume that $A$ be a subset of a $\mathbb{D}$-complete metric
space $(X,d_{\mathbb{D}}).$ Then the subspace $(A,d_{\mathbb{D}}|A\times A)$
is $\mathbb{D}$-complete if and only if $A$ is $\mathbb{D}$-closed subset of 
$X.$
\end{proposition}

\begin{proof}
Proof is along the similar lines as in [ \cite{lin}, Theorem $3.4.4$]
\end{proof}

\begin{proposition}
Let $K$ be a $\mathbb{D}$-compact subset of a $\mathbb{D}$-metric space $%
(X,d_{\mathbb{D}}).$ Then $K$ is a totally $\mathbb{D}$-bounded.

\begin{proof}
Assume that $K$ is not totally $\mathbb{D}$-bounded.

There is an $\epsilon _{\mathbb{D}}\in \mathbb{D}^{+}$\ such that no finite
collection of~$\mathbb{D}$-open balls with hyperbolic radius $\epsilon _{%
\mathbb{D}}$ covers $K.$

Now we construct a sequence $\{x_{n}\}$ in $K$ that does not have a $\mathbb{%
D}$-convergent subsequence.

Choose an arbitrary elements $x_{1}\in K.$ Since $B_{\mathbb{D}%
}(x_{1},\epsilon _{\mathbb{D}})$ does not cover $K,$ we can choose $x_{2}\in
K-B_{\mathbb{D}}(x_{1},\epsilon _{\mathbb{D}}).$

Again since $B_{\mathbb{D}}(x_{1},\epsilon _{\mathbb{D}})$ and $B_{\mathbb{D}%
}(x_{2},\epsilon _{\mathbb{D}})$ do not cover $K,$ we can choose $x_{2}\in
K-\left( B_{\mathbb{D}}(x_{1},\epsilon _{\mathbb{D}})\cup B_{\mathbb{D}%
}(x_{2},\epsilon _{\mathbb{D}})\right) .$

Continuing this way, we get a sequence $\{x_{n}\}$ such that%
\begin{equation*}
x_{n}\in K-\left( B_{\mathbb{D}}(x_{1},\epsilon _{\mathbb{D}})\cup B_{%
\mathbb{D}}(x_{2},\epsilon _{\mathbb{D}})\cup \cdots \cup B_{\mathbb{D}%
}(x_{n-1},\epsilon _{\mathbb{D}})\right) .
\end{equation*}

This means that $d_{\mathbb{D}}(x_{n},x_{m})\succeq \epsilon _{\mathbb{D}}$
for all $n,m\in 
%TCIMACRO{\U{2115} }%
%BeginExpansion
\mathbb{N}
%EndExpansion
,n\neq m,$ and hence $\{x_{n}\}$ has no $\mathbb{D}$-convergent subsequence.
\end{proof}
\end{proposition}

\begin{theorem}[Extreme Value Theorem]
\label{t evt} Let $K$ be a nonempty $\mathbb{D}$-compact subset of $(X,d_{%
\mathbb{D}})$~and $f:K\rightarrow \mathbb{D}$ be a $\mathbb{D}$-continuous
function. Let $\sup f(K)=M=M_{1}\mathbf{e}_{1}+M_{2}\mathbf{e}_{2}$ and $%
\inf f(K)=m=m_{1}\mathbf{e}_{1}+m_{2}\mathbf{e}_{2},$ then there exist $%
a,b,c,d\in K$ such that $%
p_{1}(f(a))=M_{1},~p_{2}(f(b))=M_{2},~p_{1}(f(c))=m_{1},~p_{2}(f(d))=m_{2}.$

\begin{proof}
Since we have $\sup f(K)=\sup f(K)_{\mathbf{e}_{1}}\mathbf{e}_{1}+\sup f(K)_{%
\mathbf{e}_{2}}\mathbf{e}_{2}=M_{1}\mathbf{e}_{1}+M_{2}\mathbf{e}_{2},$ we
have $M_{1}=\sup f(K)_{\mathbf{e}_{1}},$ where 
\begin{equation*}
f(K)_{\mathbf{e}_{1}}=\{x\in 
%TCIMACRO{\U{211d} }%
%BeginExpansion
\mathbb{R}
%EndExpansion
:~\text{there exists }y\in 
%TCIMACRO{\U{211d} }%
%BeginExpansion
\mathbb{R}
%EndExpansion
~\text{such that }x\mathbf{e}_{1}+y\mathbf{e}_{2}\in f(K)\}.
\end{equation*}

Thus, there exists a sequence $\{x_{n}\}~\ $in $f(K)_{\mathbf{e}_{1}}$ which
converges to $M_{1}.$

Again for each $x_{n}~\ $in $f(K)_{\mathbf{e}_{1}}~$there exists $y_{n}\in
f(K)_{\mathbf{e}_{2}}$ such that $x_{n}\mathbf{e}_{1}+y_{n}\mathbf{e}_{2}\in
f(K),~$for all $n\in \mathbb{N}.$

Again since $f(K)$ is a $\mathbb{D}$-continuous image of a $\mathbb{D}$%
-compact set, $f(K)$ is $\mathbb{D}$-compact.

Thus the sequence $\{x_{n}\mathbf{e}_{1}+y_{n}\mathbf{e}_{2}\}$ in $f(K)$
has a $\mathbb{D}$-convergent subsquence, say $\{x_{n_{k}}\mathbf{e}%
_{1}+y_{n_{k}}\mathbf{e}_{2}\}$~converging to $M_{1}\mathbf{e}_{1}+Y\mathbf{e%
}_{2}$ (say)$.$

Now, let $\alpha _{n_{k}}\in K$ be such that $f(\alpha _{n_{k}})=x_{n_{k}}%
\mathbf{e}_{1}+y_{n_{k}}\mathbf{e}_{2},~~$for all $n.$

Then also $\{\alpha _{n_{k}}\}$ has a $\mathbb{D}$-convergent subsquence,
say $\{\alpha _{n_{j}}\},$ since $K$ is $\mathbb{D}$-compact.

Let $\{\alpha _{n_{j}}\}$ converges to $a.$

Thus, $f(a)=M_{1}\mathbf{e}_{1}+Y\mathbf{e}_{2}.$

Which implies $p_{1}(f(a))=M_{1}.$

Again since, $M_{2}=\sup f(K)_{\mathbf{e}_{2}},$ where 
\begin{equation*}
f(K)_{\mathbf{e}_{2}}=\{y^{\prime }\in 
%TCIMACRO{\U{211d} }%
%BeginExpansion
\mathbb{R}
%EndExpansion
:\text{there exists}~x^{\prime }\in 
%TCIMACRO{\U{211d} }%
%BeginExpansion
\mathbb{R}
%EndExpansion
~\text{such that }x^{\prime }\mathbf{e}_{1}+y^{\prime }\mathbf{e}_{2}\in
f(K)\}.
\end{equation*}

Thus, there exists a sequence $\{y_{n}^{\prime }\}~\ $in $f(K)_{\mathbf{e}%
_{2}}$ which converges to $M_{2}.$

Again for each $y_{n}^{\prime }~\ $in $f(K)_{\mathbf{e}_{2}}~$there exists$\
x_{n}^{\prime }\in f(K)_{\mathbf{e}_{1}}$ such that $x_{n}^{\prime }\mathbf{e%
}_{1}+y_{n}^{\prime }\mathbf{e}_{2}\in f(K),~\forall n.$

Again since $f(K)$ is a $\mathbb{D}$-continuous image of a $\mathbb{D}$%
-compact set, $f(K)$ is $\mathbb{D}$-compact.

Thus the sequence $\{x_{n}^{\prime }\mathbf{e}_{1}+y_{n}^{\prime }\mathbf{e}%
_{2}\}$ in $f(K)$ has a $\mathbb{D}$-convergent subsquence, say $%
\{x_{n_{k}}^{\prime }\mathbf{e}_{1}+y_{n_{k}}^{\prime }\mathbf{e}_{2}\}$%
~converging to $X\mathbf{e}_{1}+M_{2}\mathbf{e}_{2}$ (say)$.$

Now, let $b_{n_{k}}\in K$ be such that $f(b_{n_{k}})=x_{n_{k}}^{\prime }%
\mathbf{e}_{1}+y_{n_{k}}^{\prime }\mathbf{e}_{2},~\forall n.$

Then also $\{b_{n_{k}}\}$ has a $\mathbb{D}$-convergent subsquence, say $%
\{b_{n_{j}}\},$ since $K$ is $\mathbb{D}$-compact.

Let $\{b_{n_{j}}\}$ converges to $b.$

Thus, $f(b)=X\mathbf{e}_{1}+M_{2}\mathbf{e}_{2}.$

Which implies $p_{1}(f(b))=M_{2}.$

Similarly, we can find $c,d\in K$ such that $p_{1}(f(c))=m_{1}$ and $%
~p_{2}(f(d))=m_{2}.$

This completes the proof.
\end{proof}
\end{theorem}

\begin{remark}
Note that, in the above theorem, we may not be able to find points $a,b\in K$
such that $f(a)=\sup f(K)$ and $f(b)=\inf f(K).$

For example, take $K=\left[ 0,1\right] _{\mathbb{D}}\subset \mathbb{D}$
equipped with $\mathbb{D}$-metric as in Example \ref{Example 1}. Take $f:%
\left[ 0,1\right] _{\mathbb{D}}\rightarrow \mathbb{D},$ given by $f(z)=z_{1}%
\mathbf{e}_{1}+(1-z_{1})\mathbf{e}_{2},$ where $z=z_{1}\mathbf{e}_{1}+z_{2}%
\mathbf{e}_{2}.$

Then, clearly $f$ is $\mathbb{D}$-continuous and $K$ be a $\mathbb{D}$%
-compact subset.

Now, $\sup f(K)=\mathbf{e}_{1}+\mathbf{e}_{2}=1$ and $\inf f(K)=0.\mathbf{e}%
_{1}+0.\mathbf{e}_{2}=0,$ but there do not exists $a,b\in \left[ 0,1\right]
_{\mathbb{D}}$ such that $f(a)=1$ and $f(b)=0.$
\end{remark}

\begin{theorem}[Banach's Contraction Mapping Principle]
Let $(X,d_{\mathbb{D}})$ be $\mathbb{D}$-complete metric space and $%
T:X\rightarrow X$ be a $\mathbb{D}$-contraction mapping having $\mathbb{D}$%
-Lipschitz constant $k~(0\prec k\prec 1).$ Then $T$ has a unique fixed point 
$x_{0}$, i.e. $T\left( x_{0}\right) =x_{0},$ and for each $x\in X,$ $%
\lim\limits_{n\rightarrow \infty }T^{n}(x)=x_{0}.$ Moreover%
\begin{equation*}
d_{\mathbb{D}}(T^{n}(x),x_{0})\preceq \frac{k^{n}}{1-k}d_{\mathbb{D}%
}(x,T(x)).
\end{equation*}

\begin{proof}
Since $T$ has $\mathbb{D}$-Lipschitz constant $k,$ for each $x\in X,$%
\begin{equation*}
d_{\mathbb{D}}(T(x),T^{2}(x))\preceq kd_{\mathbb{D}}(x,T(x)).
\end{equation*}

Adding $d_{\mathbb{D}}(x,T(x))$ to both sides of the above gives%
\begin{equation*}
d_{\mathbb{D}}(x,T(x))+d_{\mathbb{D}}(T(x),T^{2}(x))\preceq d_{\mathbb{D}%
}(x,T(x))+kd_{\mathbb{D}}(x,T(x)).
\end{equation*}

Then we have,%
\begin{equation*}
d_{\mathbb{D}}(x,T(x))\preceq \frac{1}{1-k}\left( d_{\mathbb{D}}(x,T(x))-d_{%
\mathbb{D}}(T(x),T^{2}(x))\right) .
\end{equation*}

Now define the function $\phi _{\mathbb{D}}:X-\mathbb{D}^{+}$ by setting $%
\phi _{\mathbb{D}}(x)=\frac{1}{1-k}d_{\mathbb{D}}(x,T(x))$ for $x\in X.$

Thus 
\begin{equation*}
d_{\mathbb{D}}(x,T(x))\preceq \phi _{\mathbb{D}}(x)-\phi _{\mathbb{D}}(T(x)),%
\text{ }x\in X.
\end{equation*}

Therefore if $x\in X$ and $m,n\in 
%TCIMACRO{\U{2115} }%
%BeginExpansion
\mathbb{N}
%EndExpansion
$ with $n<m,$%
\begin{equation*}
d_{\mathbb{D}}(T^{n}(x),T^{m+1}(x))\preceq \sum\limits_{i=n}^{m}d_{\mathbb{D}%
}(T^{i}(x),T^{i+1}(x))\preceq \phi _{\mathbb{D}}(T^{n}(x))-\phi _{\mathbb{D}%
}(T^{m+1}(x)).
\end{equation*}

In particular taking $n=1$ and letting $m\rightarrow \infty $ we have%
\begin{equation*}
\sum\limits_{i=1}^{\infty }d_{\mathbb{D}}(T^{i}(x),T^{i+1}(x))\preceq \phi _{%
\mathbb{D}}(T(x))\prec \infty _{\mathbb{D}}.
\end{equation*}

By proposition(\ref{P2.3.1}), $\{T^{n}(x)\}$ is a $\mathbb{D}$-cauchy
sequence.

Since $X$ is $\mathbb{D}$-complete, there exists $x_{0}\in X$ such that 
\begin{equation*}
\lim_{n\rightarrow \infty }T^{n}(x)=x_{0}.
\end{equation*}

Now using remark(\ref{R2.4.1}), we have%
\begin{equation*}
x_{0}=\lim_{n\rightarrow \infty }T^{n}(x)=\lim_{n\rightarrow \infty
}T^{n+1}(x)=T(x_{0}).
\end{equation*}

Thus $x_{0}$ is a fixed point of $T.$ In order to prove the uniqueness, let $%
y$ be another fixed point of $T.$ Then from the previous result we have%
\begin{equation*}
x_{0}=\lim_{n\rightarrow \infty }T^{n}(y)=y.
\end{equation*}

Now we have%
\begin{equation*}
d_{\mathbb{D}}(T^{n}(x),T^{m+1}(x))\preceq \phi _{\mathbb{D}}(T^{n}(x))-\phi
_{\mathbb{D}}(T^{m+1}(x)),
\end{equation*}

letting $m\rightarrow \infty ,$ we get%
\begin{equation*}
d_{\mathbb{D}}(T^{n}(x),x_{0})\preceq \phi _{\mathbb{D}}(T^{n}(x))=\frac{1}{%
1-k}d_{\mathbb{D}}(T^{n}(x),T^{n+1}(x)).
\end{equation*}

Since,%
\begin{equation*}
\frac{1}{1-k}d_{\mathbb{D}}(T^{n}(x),T^{n+1}(x))\preceq \frac{k^{n}}{1-k}d_{%
\mathbb{D}}(x,T(x)),
\end{equation*}

we have%
\begin{equation*}
d_{\mathbb{D}}(T^{n}(x),x_{0})\preceq \frac{k^{n}}{1-k}d_{\mathbb{D}%
}(x,T(x)).
\end{equation*}
\end{proof}
\end{theorem}

Now a natural question arises from the previous theorem. If we replaces the
sequence $\{T^{n}(x)\}$ with $\{y_{n}\}$ where $y_{0}=x$ and $y_{n+1}$ is
approximately $T(y_{n}),$ then under what condition it still be the case
that $\lim\limits_{n\rightarrow \infty }y_{n}=x_{0}?$ The answer of this
question is given in the next theorem. We only give the statement of the
theorem as it can be proved using the idempotent decomposition of hyperbolic
number in the same lines as in [ \cite{kh}, Theorem $3.2$]

\begin{theorem}
Let $(X,d_{\mathbb{D}})$ be $\mathbb{D}$-complete metric space, let $%
T:X\rightarrow X$ be a $\mathbb{D}$-contraction mapping with $\mathbb{D}$%
-Lipschitz constant $k$ and suppose $x_{0}\in X$ is the fixed point of $T.$
Let $\{\epsilon _{\mathbb{D}n}\}$ be a sequence in $\mathbb{D}^{+}$ for
which $\lim\limits_{n\rightarrow \infty }\epsilon _{\mathbb{D}n}=0,$ let $%
y_{0}\in X$ and suppose $\{y_{n}\}\subseteq X$ satisfying $d_{\mathbb{D}%
}(y_{n+1},T(y_{n}))\preceq \epsilon _{\mathbb{D}n}.$ Then 
\begin{equation*}
\lim\limits_{n\rightarrow \infty }y_{n}=x_{0}.
\end{equation*}
\end{theorem}

\begin{theorem}
Suppose $(X,d_{\mathbb{D}})$ be $\mathbb{D}$-complete metric space and $%
T:X\rightarrow X$ be a mapping for which $T^{N}$ is a $\mathbb{D}$%
-contraction mapping for some positive integer $N.$ Then $T$ has a unique
fixed point.

\begin{proof}
Proof is along the similar lines as in [ \cite{kh}, Theorem $3.3$]
\end{proof}
\end{theorem}

The next result shows that a $\mathbb{D}$-contractive mapping can also have
a fixed point in a $\mathbb{D}$-compact subspace of $\mathbb{D}$, induced
with $\mathbb{D}$-metric $d_{\mathbb{D}}$ as in Example \ref{Example 1}.

\begin{theorem}
Let $(X,d_{\mathbb{D}})$ be a $\mathbb{D}$-compact metric space such that $%
X\subseteq \mathbb{D}$ and $T:X\rightarrow X$ be a $\mathbb{D}$-contractive
mapping. Then $T$ has a unique fixed point $x_{0},$ and moreover, for each $%
x\in X,$ $\lim\limits_{n\rightarrow \infty }T^{n}(x)=x_{0}.$

\begin{proof}
Let us consider the mapping $\psi _{\mathbb{D}}:X\rightarrow \mathbb{D}^{+}$
given by%
\begin{equation*}
\psi _{\mathbb{D}}(x)=d_{\mathbb{D}}(x,T(x)),\text{ }x\in X.
\end{equation*}

Let $\psi _{\mathbb{D}}(x)=\psi _{1}(x)\mathbf{e}_{1}+\psi _{2}(x)\mathbf{e}%
_{2}.$ Then $\psi _{i}:X\rightarrow R^{+}$ and $\psi _{i}(x)=d_{i}(x,T(x)),$ 
$x\in X,$ for every $i=1,2.$

Since $d_{\mathbb{D}}=d_{1}\mathbf{e}_{1}+d_{2}\mathbf{e}_{2}$ is $\mathbb{D}
$-contractive, each $d_{i}$ is contractive mapping on $X.$

Now for each $i=1,2,$ $\psi _{i}$ is continuous and bounded below, so $\psi
_{i}$ assumes its minimum value at some point $y_{i0}\in X.$ Since $%
y_{i0}\neq T(y_{i0})$ implies%
\begin{equation*}
\psi
_{i}(T(y_{i0}))=d_{i}(T(y_{i0}),T^{2}(y_{i0}))<d_{i}(y_{i0}T(y_{i0}))=\psi
_{i}(y_{i0}),
\end{equation*}%
which is a contradiction. Hence $y_{i0}=T(y_{i0})$ and so $y_{i0}$ is a
solution of $\psi _{i}(x)=0.$

Take%
\begin{equation*}
x_{i0}=\min \{p_{i}(x):\psi _{i}(x)=0\},\text{ }i=1,2.
\end{equation*}

Now putting $x_{0}=x_{10}\mathbf{e}_{1}+x_{20}\mathbf{e}_{2},$ we get $\psi
_{\mathbb{D}}(x_{0})=0.$

Hence $T(x_{0})=x_{0}.$

Let $y_{0}$ be another fixed point of $T.$ Then $T(y_{0})=y_{0}.$ Hence 
\begin{equation*}
d_{\mathbb{D}}(x_{0},y_{0})=d_{\mathbb{D}}(T(x_{0}),T(y_{0}))\prec d_{%
\mathbb{D}}(x_{0},y_{0}),
\end{equation*}%
which is a contradiction.

Hence $T$ has unique fixed point.

Now let $x\in X$ and consider the sequence $\{d_{\mathbb{D}%
}(T^{n}(x),x_{0})\}.$ If $T^{n}(x)\neq x_{0},$%
\begin{equation*}
d_{\mathbb{D}}(T^{n+1}(x),x_{0})=d_{\mathbb{D}}(T^{n+1}(x),T(x_{0}))\prec d_{%
\mathbb{D}}(T^{n}(x),x_{0}),
\end{equation*}

so $\{d_{\mathbb{D}}(T^{n}(x),x_{0})\}$ is strictly decreasing. Consequently
the limit%
\begin{equation*}
r=\lim\limits_{n\rightarrow \infty }d_{\mathbb{D}}(T^{n}(x),x_{0})
\end{equation*}

exists and $r\in \mathbb{D}^{+}.$ Also since $X$ is compact, the sequence $%
\{T^{n}(x)\}$ has a $\mathbb{D}$-convergent subsequence $\{T^{n_{k}}(x)\},$
say $\lim\limits_{k\rightarrow \infty }T^{n_{k}}(x)=z\in X.$ Since $%
\{T^{n}(x)\}$ is decreasing,%
\begin{equation*}
r=d_{\mathbb{D}}(z,x_{0})=\lim\limits_{k\rightarrow \infty }d_{\mathbb{D}%
}(T^{n_{k}}(x),x_{0})=\lim\limits_{k\rightarrow \infty }d_{\mathbb{D}%
}(T^{n_{k}+1}(x),x_{0})=\lim\limits_{k\rightarrow \infty }d_{\mathbb{D}%
}(T(z),x_{0}).
\end{equation*}

But if $z\neq x_{0},$ then $d_{\mathbb{D}}(T(z),x_{0})=d_{\mathbb{D}%
}(T(z),T(x_{0}))\prec d_{\mathbb{D}}(z,x_{0}).$ This proves that any $%
\mathbb{D}$-convergent subsequence of $\{T^{n}(x)\}$ $\mathbb{D}$-converges
to $x_{0},$ so it must be the case that $\lim\limits_{n\rightarrow \infty
}T^{n}(x)=x_{0}.$
\end{proof}
\end{theorem}

Our next task is to construct a $\mathbb{D}$-metric on the space of $\mathbb{%
D}$-continuous functions. To proceed further, first we consider the space of 
$\mathbb{D}$-bounded functions.

Let $(X,d_{\mathbb{D}})$ and $(Y,\rho _{\mathbb{D}})$ be two $\mathbb{D}$%
-metric spaces. A function $f:X\rightarrow Y$ is $\mathbb{D}$-bounded if the
range set $f(X)=\{f(x):x\in X\}\subset Y$ is $\mathbb{D}$-bounded with
respect to the $\mathbb{D}$-metric $\rho _{\mathbb{D}},$ i.e. for all $%
u,v\in X,$ there exists $M\in \mathbb{D}^{+},$ such that $\rho _{\mathbb{D}%
}(f(u),f(v))\preceq M.$

Let%
\begin{equation*}
B_{\mathbb{D}}(X,Y)=\{f:X\rightarrow Y~|~f\text{ is }\mathbb{D}-\text{bounded%
}\}
\end{equation*}%
be the collection of all $\mathbb{D}$-bounded functions from $X$ to $Y.$ We
shall turn $B_{\mathbb{D}}(X,Y)$ into a $\mathbb{D}$-metric space by
introducing a $\mathbb{D}$-metric $\sigma _{\mathbb{D}}.$

\begin{proposition}
Let$~(X,d_{\mathbb{D}})$ and $(Y,\rho _{\mathbb{D}})$ be two $\mathbb{D}$%
-metric spaces. Then 
\begin{equation*}
\sigma _{\mathbb{D}}(f,g)=\sup \{\rho _{\mathbb{D}}(f(x),g(x))~|~x\in X\}
\end{equation*}%
defines a $\mathbb{D}$-metric $\sigma _{\mathbb{D}}$ on $B_{\mathbb{D}%
}(X,Y). $

\begin{proof}
It is clear form the definition $\sigma _{\mathbb{D}}(f,g)\succeq 0$ and $%
\sigma _{\mathbb{D}}(f,g)=0$ if and only if $f=g.$

Also $\sigma _{\mathbb{D}}(f,g)=\sigma _{\mathbb{D}}(g,f),$ for all $f,g\in
B_{\mathbb{D}}(X,Y)$

Again for all $x\in X,$%
\begin{equation*}
\rho _{\mathbb{D}}(f(x),g(x))\preceq \rho _{\mathbb{D}}(f(x),h(x))+\rho _{%
\mathbb{D}}(h(x),g(x))\preceq \sigma _{\mathbb{D}}(f,h)+\sigma _{\mathbb{D}%
}(h,g),
\end{equation*}%
and taking supremum over all $x\in X,$ we get%
\begin{equation*}
\sigma _{\mathbb{D}}(f,g)\preceq \sigma _{\mathbb{D}}(f,h)+\sigma _{\mathbb{D%
}}(h,g).
\end{equation*}

Thus, $\sigma _{\mathbb{D}}$ defines a $\mathbb{D}$-metric on $B_{\mathbb{D}%
}(X,Y).$\bigskip
\end{proof}
\end{proposition}

\begin{theorem}
\label{t comp} Let$\ (X,d_{\mathbb{D}})$ and $(Y,\rho _{\mathbb{D}})$ be two 
$\mathbb{D}$-metric spaces and assume that $(Y,\rho _{\mathbb{D}})$ is $%
\mathbb{D}$-complete. Then $\left( B_{\mathbb{D}}(X,Y),\sigma _{\mathbb{D}%
}\right) $ is also $\mathbb{D}$-complete.
\end{theorem}

\begin{proof}
Proof is along the similar lines as in [ \cite{lin}, Theorem $4.5.3$]
\end{proof}

Now we shall turn to the spaces of $\mathbb{D}$-continuous functions.

As before we assume that $(X,d_{\mathbb{D}})$ and $(Y,\rho _{\mathbb{D}})$
be two $\mathbb{D}$-metric spaces. We define 
\begin{equation*}
Cb_{\mathbb{D}}(X,Y)=\{f:X\rightarrow Y~|~f\text{ is }\mathbb{D}\text{%
-continuous and }\mathbb{D}\text{-bounded}\}
\end{equation*}%
to be the collection of all $\mathbb{D}$-bounded $\mathbb{D}$-continuous
functions from $X$ to $Y.$

Since $Cb_{\mathbb{D}}(X,Y)$ is a subset of $B_{\mathbb{D}}(X,Y),$ the metric%
\begin{equation*}
\sigma _{\mathbb{D}}(f,g)=\sup \{\rho _{\mathbb{D}}(f(x),g(x))~|~x\in X\}
\end{equation*}%
that we introduced on $B_{\mathbb{D}}(X,Y)$ is also a $\mathbb{D}$-metric on 
$Cb_{\mathbb{D}}(X,Y).$

\begin{proposition}
$Cb_{\mathbb{D}}(X,Y)$ is a $\mathbb{D}$-closed subset of $B_{\mathbb{D}%
}(X,Y).$

\begin{proof}
It suffices to show that if $\{f_{n}\}$ is a sequence in $Cb_{\mathbb{D}%
}(X,Y)$ that $\mathbb{D}$-converges to an element $f\in B(X,Y),$ then $f\in
Cb_{\mathbb{D}}(X,Y).$

By the definition of the $\mathbb{D}$-metric $\sigma _{\mathbb{D}},$ it is
clear that $\{f_{n}\}$ $\mathbb{D}$-converges uniformly to $f.$ Again since $%
f_{n}$ is $\mathbb{D}$-bounded $\mathbb{D}$-continuous for all $n\in 
%TCIMACRO{\U{2115} }%
%BeginExpansion
\mathbb{N}
%EndExpansion
,$ it can be shown that $f$ is also $\mathbb{D}$-bounded $\mathbb{D}$%
-continuous and hence $f\in Cb_{\mathbb{D}}(X,Y).$
\end{proof}
\end{proposition}

\begin{theorem}
Let$\ (X,d_{\mathbb{D}})$ and $(Y,\rho _{\mathbb{D}})$ be two $\mathbb{D}$%
-metric spaces and assume that $(Y,\rho _{\mathbb{D}})$ is $\mathbb{D}$%
-complete. Then $(Cb_{\mathbb{D}}(X,Y),\sigma _{\mathbb{D}})$ is also $%
\mathbb{D}$-complete.

\begin{proof}
Recall from Proposition \ref{p cplt} that a $\mathbb{D}$-closed subspace of
a $\mathbb{D}$-complete space is itself $\mathbb{D}$-complete. Since $B_{%
\mathbb{D}}(X,Y)$ is $\mathbb{D}$-complete by Theorem \ref{t comp}, and $Cb_{%
\mathbb{D}}(X,Y)$ is $\mathbb{D}$-closed subset of $B_{\mathbb{D}}(X,Y)$ by
the above proposition, it follows that $Cb_{\mathbb{D}}(X,Y)$ is $\mathbb{D}$%
-complete.
\end{proof}
\end{theorem}

\begin{proposition}
Let$\ (X,d_{\mathbb{D}})$ and $(Y,\rho _{\mathbb{D}})$ be two $\mathbb{D}$%
-metric spaces and assume that $X$ is $\mathbb{D}$-compact. Then all $%
\mathbb{D}$-continuous functions from $X$ to $Y$ are $\mathbb{D}$-bounded.

\begin{proof}
Assume that $f:X\longrightarrow Y$ is $\mathbb{D}$-continuous, and pick a
point $a\in X.$ It suffices to prove that the function 
\begin{equation*}
h(x)=\rho _{\mathbb{D}}(f(x),f(a))
\end{equation*}%
is $\mathbb{D}$-bounded, and this will follow from the Extreme Value Theorem %
\ref{t evt} if we can show that it is $\mathbb{D}$-continuous.

Since%
\begin{equation*}
\delta _{\mathbb{D}}\left( h(x),h(y)\right) =\delta _{\mathbb{D}}\left( \rho
_{\mathbb{D}}(f(x),f(a)),\rho _{\mathbb{D}}(f(y),f(a))\right) \preceq \rho _{%
\mathbb{D}}(f(x),f(y)),
\end{equation*}%
$($where $\delta _{\mathbb{D}}$ is the $\mathbb{D}$-metric on $\mathbb{D)}$

And since $f$ is $\mathbb{D}$-continuous, so is $h$.
\end{proof}
\end{proposition}

Now if we define%
\begin{equation*}
C_{\mathbb{D}}(X,Y)=\{f:X\rightarrow Y~|~f\text{ is }\mathbb{D}-\text{%
continuous}\}
\end{equation*}%
where $X$ is $\mathbb{D}$-compact, the above proposition tells us that $Cb_{%
\mathbb{D}}(X,Y)~$and\ $C_{\mathbb{D}}(X,Y)$ coincide. The following theorem
sums up the results above for $X$ $\mathbb{D}$-compact.

\begin{theorem}
Let$\ (X,d_{\mathbb{D}})$ and $(Y,\rho _{\mathbb{D}})$ be two $\mathbb{D}$%
-metric spaces and assume that $X$ is $\mathbb{D}$-compact, then%
\begin{equation*}
\sigma _{\mathbb{D}}(f,g)=\sup \{\rho _{\mathbb{D}}(f(x),g(x))~|~x\in X\}
\end{equation*}%
defines a $\mathbb{D}$-metric $\sigma _{\mathbb{D}}$ on $C_{\mathbb{D}%
}(X,Y).~$Also if $(Y,\rho _{\mathbb{D}})$ is $\mathbb{D}$-complete, so is $%
\left( C_{\mathbb{D}}(X,Y),\sigma _{\mathbb{D}}\right) .$
\end{theorem}

\section{Examples}

In this section we give some examples supporting our results.

\begin{example}
\label{Example 1} Let $X=\mathbb{D}~\,$and $d_{\mathbb{D}}:\mathbb{D}\times 
\mathbb{D}\longrightarrow \mathbb{D}_{0}^{+}$ be given by 
\begin{equation*}
d_{\mathbb{D}}(\alpha ,\beta )=\left\vert \alpha _{1}-\beta _{1}\right\vert 
\mathbf{e}_{1}+\left\vert \alpha _{2}-\beta _{2}\right\vert \mathbf{e}_{2},
\end{equation*}%
where $\alpha =\alpha _{1}\mathbf{e}_{1}+\alpha _{2}\mathbf{e}_{2},~\beta
=\beta _{1}\mathbf{e}_{1}+\beta _{2}\mathbf{e}_{2}\in \mathbb{D}$ and $%
\left\vert .\right\vert $ is the real modulus.

Then clearly $\left\vert \alpha _{1}-\beta _{1}\right\vert ,\left\vert
\alpha _{2}-\beta _{2}\right\vert \geq 0$ and thus $d_{\mathbb{D}}(\alpha
,\beta )\succeq 0\,,\ \forall \alpha ,\beta \in \mathbb{D}.$

Also 
\begin{equation*}
d_{\mathbb{D}}(\alpha ,\beta )=0\Longleftrightarrow \left\vert \alpha
_{1}-\beta _{1}\right\vert =0=\left\vert \alpha _{2}-\beta _{2}\right\vert
\Longleftrightarrow \alpha _{1}=\beta _{1}\text{ and }\alpha _{2}=\beta
_{2}\Longleftrightarrow \alpha =\beta .
\end{equation*}

Again, it is clear from the definition 
\begin{equation*}
d_{\mathbb{D}}(\alpha ,\beta )=d_{\mathbb{D}}(\beta ,\alpha ),\ \forall
\alpha ,\beta \in \mathbb{D}.
\end{equation*}

Next let $\alpha =\alpha _{1}\mathbf{e}_{1}+\alpha _{2}\mathbf{e}_{2},~\beta
=\beta _{1}\mathbf{e}_{1}+\beta _{2}\mathbf{e}_{2},\gamma =\gamma _{1}%
\mathbf{e}_{1}+\gamma _{2}\mathbf{e}_{2}\in \mathbb{D}$.

Then, 
\begin{eqnarray*}
d_{\mathbb{D}}(\alpha ,\beta ) &=&\left\vert \alpha _{1}-\beta
_{1}\right\vert \mathbf{e}_{1}+\left\vert \alpha _{2}-\beta _{2}\right\vert 
\mathbf{e}_{2} \\
&=&\left\vert (\alpha _{1}-\gamma _{1})+(\gamma _{1}-\beta _{1})\right\vert 
\mathbf{e}_{1}+\left\vert (\alpha _{2}-\gamma _{2})+(\gamma _{2}-\beta
_{2})\right\vert \mathbf{e}_{2} \\
&\preceq &\left\vert \alpha _{1}-\gamma _{1}\right\vert \mathbf{e}%
_{1}+\left\vert \gamma _{1}-\beta _{1}\right\vert \mathbf{e}_{1}+\left\vert
\alpha _{2}-\gamma _{2}\right\vert \mathbf{e}_{2}+\left\vert \gamma
_{2}-\beta _{2}\right\vert \mathbf{e}_{2} \\
&=&\left\vert \alpha _{1}-\gamma _{1}\right\vert \mathbf{e}_{1}+\left\vert
\alpha _{2}-\gamma _{2}\right\vert \mathbf{e}_{2}+\left\vert \gamma
_{1}-\beta _{1}\right\vert \mathbf{e}_{1}+\left\vert \alpha _{2}-\gamma
_{2}\right\vert \mathbf{e}_{2}+\left\vert \gamma _{2}-\beta _{2}\right\vert 
\mathbf{e}_{2} \\
&=&d_{\mathbb{D}}(\alpha ,\gamma )+d_{\mathbb{D}}(\gamma ,\alpha ).
\end{eqnarray*}

Thus $d_{\mathbb{D}}(\alpha ,\beta )=\left\vert \alpha _{1}-\beta
_{1}\right\vert \mathbf{e}_{1}+\left\vert \alpha _{2}-\beta _{2}\right\vert 
\mathbf{e}_{2}$ is a hyperbolic valued metric on $\mathbb{D}$.

Further $\left( \mathbb{D},d_{\mathbb{D}}\right) $ ia $\mathbb{D}$-complete
and every $\mathbb{D}$-closed and $\mathbb{D}$-bounded subset of $\mathbb{D}$
is $\mathbb{D}$-compact.

Now let $\alpha =\alpha _{1}\mathbf{e}_{1}+\alpha _{2}\mathbf{e}_{2}\in 
\mathbb{D}$ and$~r=r_{1}\mathbf{e}_{1}+r_{2}\mathbf{e}_{2}\in \mathbb{D}^{+}$%
. Then%
\begin{eqnarray*}
B_{\mathbb{D}}(\alpha ;r) &=&\{x=x_{1}\mathbf{e}_{1}+x_{2}\mathbf{e}_{2}\in 
\mathbb{D}:\left\vert \alpha _{1}-x_{1}\right\vert \mathbf{e}_{1}+\left\vert
\alpha _{2}-x_{2}\right\vert \mathbf{e}_{2}\prec r_{1}\mathbf{e}_{1}+r_{2}%
\mathbf{e}_{2}\} \\
&=&\{x=x_{1}\mathbf{e}_{1}+x_{2}\mathbf{e}_{2}\in \mathbb{D}:\left\vert
\alpha _{1}-x_{1}\right\vert <r_{1},\left\vert \alpha _{2}-x_{2}\right\vert
<r_{2}\},
\end{eqnarray*}%
which is a open square with vertices at $\left( \alpha _{1}+r_{1}\right) 
\mathbf{e}_{1}+\left( \alpha _{2}+r_{2}\right) \mathbf{e}_{2},$ $\left(
\alpha _{1}+r_{1}\right) \mathbf{e}_{1}+\left( \alpha _{2}-r_{2}\right) 
\mathbf{e}_{2},$ $\left( \alpha _{1}-r_{1}\right) \mathbf{e}_{1}+\left(
\alpha _{2}+r_{2}\right) \mathbf{e}_{2}$ and $\left( \alpha
_{1}-r_{1}\right) \mathbf{e}_{1}+\left( \alpha _{2}-r_{2}\right) \mathbf{e}%
_{2};$%
\begin{eqnarray*}
\overline{B_{\mathbb{D}}}(\alpha ;r) &=&\{x=x_{1}\mathbf{e}_{1}+x_{2}\mathbf{%
e}_{2}\in \mathbb{D}:\left\vert \alpha _{1}-x_{1}\right\vert \mathbf{e}%
_{1}+\left\vert \alpha _{2}-x_{2}\right\vert \mathbf{e}_{2}\preceq r_{1}%
\mathbf{e}_{1}+r_{2}\mathbf{e}_{2}\} \\
&=&\{x=x_{1}\mathbf{e}_{1}+x_{2}\mathbf{e}_{2}\in \mathbb{D}:\left\vert
\alpha _{1}-x_{1}\right\vert \leq r_{1},\left\vert \alpha
_{2}-x_{2}\right\vert \leq r_{2}\},
\end{eqnarray*}%
which is a closed square with vertices at $\left( \alpha _{1}+r_{1}\right) 
\mathbf{e}_{1}+\left( \alpha _{2}+r_{2}\right) \mathbf{e}_{2},$ $\left(
\alpha _{1}+r_{1}\right) \mathbf{e}_{1}+\left( \alpha _{2}-r_{2}\right) 
\mathbf{e}_{2},$ $\left( \alpha _{1}-r_{1}\right) \mathbf{e}_{1}+\left(
\alpha _{2}+r_{2}\right) \mathbf{e}_{2}$ and $\left( \alpha
_{1}-r_{1}\right) \mathbf{e}_{1}+\left( \alpha _{2}-r_{2}\right) \mathbf{e}%
_{2};$ and 
\begin{align*}
S_{\mathbb{D}}(\alpha ;r)& =\{x=x_{1}\mathbf{e}_{1}+x_{2}\mathbf{e}_{2}\in 
\mathbb{D}:\left\vert \alpha _{1}-x_{1}\right\vert \mathbf{e}_{1}+\left\vert
\alpha _{2}-x_{2}\right\vert \mathbf{e}_{2}=r_{1}\mathbf{e}_{1}+r_{2}\mathbf{%
e}_{2}\} \\
& =\{x=x_{1}\mathbf{e}_{1}+x_{2}\mathbf{e}_{2}\in \mathbb{D}:\left\vert
\alpha _{1}-x_{1}\right\vert =r_{1},\left\vert \alpha _{2}-x_{2}\right\vert
=r_{2}\}. \\
& =\{\left( \alpha _{1}+r_{1}\right) \mathbf{e}_{1}+\left( \alpha
_{2}+r_{2}\right) \mathbf{e}_{2},\left( \alpha _{1}+r_{1}\right) \mathbf{e}%
_{1}+\left( \alpha _{2}-r_{2}\right) \mathbf{e}_{2},\left( \alpha
_{1}-r_{1}\right) \mathbf{e}_{1} \\
& +\left( \alpha _{2}+r_{2}\right) \mathbf{e}_{2},\left( \alpha
_{1}-r_{1}\right) \mathbf{e}_{1}+\left( \alpha _{2}-r_{2}\right) \mathbf{e}%
_{2}\}
\end{align*}%
which consist of only four points.

So, in this situation, unlike $%
%TCIMACRO{\U{211d} }%
%BeginExpansion
\mathbb{R}
%EndExpansion
^{2}$ equipped with usual metric,%
\begin{equation*}
\overline{B_{\mathbb{D}}(a;r)}-B_{\mathbb{D}}(\alpha ;r)\neq S_{\mathbb{D}%
}(\alpha ;r).
\end{equation*}
\end{example}

\begin{example}
Let $X=\mathbb{B%
%TCIMACRO{\U{2102} }%
%BeginExpansion
\mathbb{C}
%EndExpansion
},\,$and $d_{\mathbb{D}}:\mathbb{B%
%TCIMACRO{\U{2102} }%
%BeginExpansion
\mathbb{C}
%EndExpansion
}\times \mathbb{B%
%TCIMACRO{\U{2102} }%
%BeginExpansion
\mathbb{C}
%EndExpansion
}\longrightarrow \mathbb{D}_{0}^{+}$ be given by $d_{\mathbb{D}}(\alpha
,\beta )=\left\vert \alpha -\beta \right\vert _{k},~$where $\alpha ,~\beta
\in \mathbb{B%
%TCIMACRO{\U{2102} }%
%BeginExpansion
\mathbb{C}
%EndExpansion
~}$and $\left\vert .\right\vert _{k}$ is the hyperbolic valued modulus on
bicomplex number.

Here we shall show the triangle Inequality for $d_{\mathbb{D}}.$

Let $\alpha ,\beta ,\gamma \in \mathbb{B%
%TCIMACRO{\U{2102} }%
%BeginExpansion
\mathbb{C}
%EndExpansion
}$ with$~$%
\begin{eqnarray*}
\alpha  &=&\alpha _{1}\mathbf{e}_{1}+\alpha _{2}\mathbf{e}_{2}=(\alpha
_{11}+i.\alpha _{12})\mathbf{e}_{1}+(\alpha _{21}+i.\alpha _{22})\mathbf{e}%
_{2}, \\
~\beta  &=&\beta _{1}\mathbf{e}_{1}+\beta _{2}\mathbf{e}_{2}=(\beta
_{11}+i.\beta _{12})\mathbf{e}_{1}+(\beta _{21}+i.\beta _{22})\mathbf{e}_{2},
\\
\gamma  &=&\gamma _{1}\mathbf{e}_{1}+\gamma _{2}\mathbf{e}_{2}=(\gamma
_{11}+i.\gamma _{12})\mathbf{e}_{1}+(\gamma _{21}+i.\gamma _{22})\mathbf{e}%
_{2}.
\end{eqnarray*}%
be their idempotent representation. Then%
\begin{align*}
d_{\mathbb{D}}(\alpha ,\gamma )+d_{\mathbb{D}}(\gamma ,\beta )& =\left\vert
\alpha -\gamma \right\vert _{k}+\left\vert \gamma -\beta \right\vert _{k} \\
& =\left\vert \alpha _{1}\mathbf{e}_{1}+\alpha _{2}\mathbf{e}_{2}-\gamma _{1}%
\mathbf{e}_{1}-\gamma _{2}\mathbf{e}_{2}\right\vert _{k}+\left\vert \gamma
_{1}\mathbf{e}_{1}+\gamma _{2}\mathbf{e}_{2}-\beta _{1}e_{1}-\beta _{2}%
\mathbf{e}_{2}\right\vert _{k} \\
& =\left\vert (\alpha _{1}-\gamma _{1})\mathbf{e}_{1}+(\alpha _{2}-\gamma
_{2})\mathbf{e}_{2}\right\vert _{k}+\left\vert (\gamma _{1}-\beta _{1})%
\mathbf{e}_{1}+(\gamma _{2}-\beta _{2})\mathbf{e}_{2}\right\vert _{k} \\
& =\left\vert \alpha _{1}-\gamma _{1}\right\vert \mathbf{e}_{1}+\left\vert
\alpha _{2}-\gamma _{2}\right\vert \mathbf{e}_{2}+\left\vert \gamma
_{1}-\beta _{1}\right\vert \mathbf{e}_{1}+\left\vert \gamma _{2}-\beta
_{2}\right\vert \mathbf{e}_{2},\text{ } \\
& \text{where }\left\vert .\right\vert \text{ is modulus over }\mathbb{%
%TCIMACRO{\U{2102} }%
%BeginExpansion
\mathbb{C}
%EndExpansion
}(\mathbf{i}) \\
& =\left( \left\vert \alpha _{1}-\gamma _{1}\right\vert +\left\vert \gamma
_{1}-\beta _{1}\right\vert \right) \mathbf{e}_{1}+\left( \left\vert \alpha
_{2}-\gamma _{2}\right\vert +\left\vert \gamma _{2}-\beta _{2}\right\vert
\right) \mathbf{e}_{2} \\
& =\left( \sqrt{\left( \alpha _{11}-\gamma _{11}\right) ^{2}+\left( \alpha
_{12}-\gamma _{12}\right) ^{2}}+\sqrt{\left( \gamma _{11}-\beta _{11}\right)
^{2}+\left( \gamma _{12}-\beta _{12}\right) ^{2}}\right) \mathbf{e}_{1} \\
& +\left( \sqrt{\left( \alpha _{21}-\gamma _{21}\right) ^{2}+\left( \alpha
_{22}-\gamma _{22}\right) ^{2}}+\sqrt{\left( \gamma _{21}-\beta _{21}\right)
^{2}+\left( \gamma _{22}-\beta _{22}\right) ^{2}}\right) \mathbf{e}_{2} \\
& \preceq \left( \sqrt{\left( \alpha _{11}-\beta _{11}\right) ^{2}+\left(
\alpha _{12}-\beta _{12}\right) ^{2}}\right) \mathbf{e}_{1}+\left( \sqrt{%
\left( \alpha _{21}-\beta _{21}\right) ^{2}+\left( \alpha _{22}-\beta
_{22}\right) ^{2}}\right) \mathbf{e}_{2} \\
& =\left\vert \alpha _{1}-\beta _{1}\right\vert \mathbf{e}_{1}+\left\vert
\alpha _{2}-\beta _{2}\right\vert \mathbf{e}_{2} \\
& =\left\vert \alpha -\beta \right\vert _{k} \\
& =d_{\mathbb{D}}(\alpha ,\beta )
\end{align*}%
which shows the triangle inequality.

So $d_{\mathbb{D}}(\alpha ,\beta )=\left\vert \alpha -\beta \right\vert _{k}$
is a hyperbolic valued metric on $\mathbb{B%
%TCIMACRO{\U{2102} }%
%BeginExpansion
\mathbb{C}
%EndExpansion
}.$
\end{example}

\begin{example}
Let $d_{1},~d_{2}:$ $\mathbb{%
%TCIMACRO{\U{2102} }%
%BeginExpansion
\mathbb{C}
%EndExpansion
}\times \mathbb{%
%TCIMACRO{\U{2102} }%
%BeginExpansion
\mathbb{C}
%EndExpansion
}\rightarrow \lbrack 0,\infty )~$be two real valued metrices on $\mathbb{%
%TCIMACRO{\U{2102} }%
%BeginExpansion
\mathbb{C}
%EndExpansion
}.$ Take, $d_{\mathbb{D}}:\mathbb{B%
%TCIMACRO{\U{2102} }%
%BeginExpansion
\mathbb{C}
%EndExpansion
}\times \mathbb{B%
%TCIMACRO{\U{2102} }%
%BeginExpansion
\mathbb{C}
%EndExpansion
}\longrightarrow \mathbb{D}_{0}^{+}$ given by%
\begin{equation*}
d_{\mathbb{D}}(x,y)=d_{1}(x_{1},y_{1})\mathbf{e}_{1}+d_{2}(x_{2},y_{2})%
\mathbf{e}_{2},
\end{equation*}%
where $x=x_{1}\mathbf{e}_{1}+x_{2}\mathbf{e}_{2},~y=y_{1}\mathbf{e}_{1}+y_{2}%
\mathbf{e}_{2}\in \mathbb{D},~$then $d_{\mathbb{D}}$ is a hyperbolic valued
metric on $\mathbb{B%
%TCIMACRO{\U{2102} }%
%BeginExpansion
\mathbb{C}
%EndExpansion
}$..
\end{example}


\begin{thebibliography}{9}
\bibitem{alp} Alpay, D., Luna-Elizarrar\'{a}s, M.E., Shapiro, M., Struppa,
D.C., Basics of Functional Analysis with Bicomplex Scalars and Bicomplex
Schur Analysis, Springer International Publishing, Cham, Switzerland (2014).

\bibitem{bil} Bilgin, M., Ersoy, S., Algebraic Properties of Bihyperbolic
Numbers, Adv. Appl. Clifford Algebras, 30(13), 01-17 (2020).

\bibitem{gh} Ghosh, C., Biswas, S., Yasin, T., Hyperbolic valued signed
measure, Int. J. Math. Trends Technol. 55(7), 515--522 (2018).

\bibitem{kh} Khamsi, M.A., Kirk, W.A., An Introduction to Metric Spaces and
Fixed Point Theory, Pure and Applied Mathematics, Wiley-Interscience, New
York, NY, USA, 2001.

\bibitem{ku} Kumar, R. and Saini, H., Topological bicomplex modules, Adv.
Appl. Clifford Algebras, 26(4), 1249--1270 (2016).

\bibitem{lin} Lindstr\o m, T.L., Spaces An Introduction to Real Analysis,
American Mathematical Society (2017).

\bibitem{sai} Saini, H., Sharma, A., Kumar, R., Some Fundamental Theorems of
Functional Analysis with Bicomplex and Hyperbolic Scalars, Adv. Appl.
Clifford Algebras, 30(66), 01-23 (2020).

\bibitem{seg} Segre, C., Le rappresentazioni reali delle forme complesse e
gli enti iperalgebrici (The real representation of complex elements and
hyperalgebraic entities), Math. Ann., 40, 413-467 (1892).

\bibitem{tell} Tellez-Sanchez, G.Y., Bory-Reyes, J., Generalized Iterated
Function Systems on Hyperbolic Numbers Plane, Fractals, 27(04),
doi:10.1142/S0218348X19500452 (2019).
\end{thebibliography}
\end{document}